\documentclass[12pt]{amsart}
\usepackage{amsmath}
\usepackage{amssymb}
\usepackage{amsthm}
\usepackage{array}
\usepackage{xy}
\usepackage[pdftex]{graphicx}
\usepackage{hyperref}
\usepackage{color}
\usepackage{transparent}
\usepackage{latexsym}
\usepackage{ulem}
\usepackage[margin=1.5in]{geometry}
\usepackage{amsmath,scalefnt}
\usepackage{amsthm}
\usepackage{amssymb}
\usepackage{graphicx}
\usepackage{setspace}
\usepackage{hyperref}

\numberwithin{equation}{section}

\newtheorem{thm}{Theorem}[section]

\newtheorem{prop}[thm]{Proposition}

\theoremstyle{plain}

\theoremstyle{plain}

\theoremstyle{definition}

\theoremstyle{remark}

\numberwithin{theorem}{section}
\numberwithin{equation}{section}
\numberwithin{figure}{section}

\begin{document}
\title[Non-transversal intersection]{Non-transversal intersection of the free and fixed boundary in the mean-field theory of superconductivity}

\author[Emanuel Indrei]{Emanuel Indrei}


\def\signei{\bigskip\begin{center} {\sc Emanuel Indrei\par\vspace{3mm}Department of Mathematics\\  
Purdue University\\
West Lafayette, IN 47907, USA\\
email:} {\tt eindrei@purdue.edu}
\end{center}}

\makeatletter
\def\blfootnote{\xdef\@thefnmark{}\@footnotetext}
\makeatother

\date{}

\maketitle

\begin{abstract}
Non-transversal intersection of the free and fixed boundary is shown to hold and a classification of blow-up solutions is given for obstacle problems generated by fully nonlinear uniformly elliptic operators in two dimensions which appear in the mean-field theory of superconducting vortices. 
\end{abstract}

\section{Introduction} 

Strong $L^2$-solutions are considered for the following PDE  

\begin{equation}  \label{eqF}
\begin{cases}
F(D^{2}u)=\chi_\Omega & \text{a.e. in }B_{1}^{+}\\
u=0 & \text{on }B'_{1}
\end{cases}
\end{equation}
\vskip .2in 
\noindent where $u \in W^{2,2}(B_{1}^{+})$, $F$ is a convex $C^1$ fully nonlinear uniformly elliptic operator, $\Omega \subset \mathbb{R}^2$ is an (a priori unknown) open set, and the free boundary is $\Gamma=\partial \Omega \cap B_1^+$. The assumptions imply $u \in W^{2,p}(B_{1}^{+})$ for all $p<\infty$ and $u$ satisfies \eqref{eqF} in the viscosity sense \cite{MR1376656}. Equations of the form 
$$
F(D^2u,x)=g(x,u)\chi_{\{\nabla u \neq 0\}}
$$
have been studied in \cite{MR1897393}; an example is the stationary equation for the mean-field theory of superconducting vortices when the scalar stream is a function of the scalar magnetic potential \cite{MR1349309, MR1388106, MR1664550}.

Problems of interest are the endpoint $W^{2,\infty}$ regularity estimates of the solution and the regularity of $\Gamma$. The regularity of the solution enables convergence of re-scalings to solutions in half-spaces and a classification then yields information on the geometry of the free boundary near contact points in the sense that in some configurations the intersection occurs non-transversally. Thus the free normal points in the $x_n$-direction at the contact point. The endpoint regularity is delicate since it is sensitive to the sign of the right-hand-side: there exist solutions to 
\begin{equation}  \label{eqF2}
\begin{cases}
\Delta u=-\chi_{u>0} & \text{a.e. in }B_{1}\\
u=g & \text{on }\partial B_{1}
\end{cases}
\end{equation}

\noindent for $g \in C^\infty(\partial B_1)$ which belong to $C^{1,\alpha}\setminus C^{1,1}$ for all $\alpha \in (0,1)$ \cite{MR2289547}. Up to the boundary and interior $C^{1,1}$ estimates were obtained in \cite{MR3513142} and \cite{MR3542613, MR3198649}, respectively, for more general versions of \eqref{eqF} (valid also in higher dimensions; see also \cite{LPS} for applications to double obstacle problems). 

The author proved $C^1$ regularity of the free boundary for non-negative solutions at contact points for 
$$\Omega = (\{\nabla u \neq 0\} \cup \{u \neq 0\}) \cap \{x_n>0\}$$ 
without density assumptions and non-transversal intersection was proved without a sign assumption on the solution \cite{I}. If $F(M)=\text{tr}(M)$ this problem was investigated in \cite{MR1950478, MR2281197} and in \cite{MR2180300} when the $\{u \neq 0\}$ term is removed, cf. \cite{MR2029533}.

The class of solutions for which $||u||_{L^\infty(B_1^+)} \le M$ is denoted by $P_1^+(0,M, \Omega)$. In what follows, tangential touch is shown to hold for 
$$\Omega = \{\nabla u \neq 0\} \cap \{x_2>0\}\subset \mathbb{R}_+^2$$ and a classification of blow-up solutions is given.

\begin{thm} \label{tt}
There exists $r_0>0$ and a modulus of continuity $\omega$ such that 
$$\Gamma(u) \cap B_{r_0}^+ \subset \{x=(x_1,x_2): x_2 \le \omega(|x_1|)|x_1|\}$$ for all $u \in P_1^+(0,M, \Omega)$ provided $0 \in \overline{\Gamma(u)}$.  
\end{thm}

\begin{thm} \label{cul}
If $u \in P_1^+(0,M, \Omega)$, $0 \in \overline{\{u \neq 0\}}$ and $\nabla u(0)=0$, then the blow-up limit of $u$ at the origin has the form $$u_0(x)=ax_1x_2+bx_2^2$$ for $a, b \in \mathbb{R}$.
\end{thm}

\section{Preliminaries} \label{pre}
$F$ is assumed to satisfy the following structural conditions.
\begin{itemize}
\item $F(0)=0$.
\item $F$ is uniformly elliptic with ellipticity constants $\lambda_{0}$, $\lambda_{1}>0$
such that
$$
\mathcal{P}^{-}(M-N)\le F(M)-F(N)\le\mathcal{P}^{+}(M-N),
$$
where $M$ and $N$ are symmetric matrices and $\mathcal{P}^{\pm}$
are the Pucci operators
$$
\mathcal{P}^{-}(M)=\inf_{\lambda_{0} \le N\le\lambda_{1}} \text{tr}(NM),\qquad\mathcal{P}^{+}(M)=\sup_{\lambda_{0}\le N\le\lambda_{1}}\text{tr} (NM).
$$
\item $F$ is convex and $C^1$.
\end{itemize}

Let $\Omega$ be an open set. A continuous function $u$ belongs to $P_r^+(0,M, \Omega)$ if $u$ satisfies in the viscosity sense:\\

\noindent 1. $F(D^2 u)=\chi_\Omega$ a.e. in $B_r^+$;\\
2. $||u||_{L^\infty(B_r^+)} \le M$;\\
3. $u=0$ on $\{x_n=0\} \cap \overline{B_1^+}=:B'_{r}$.\\

Furthermore, given $u \in P_r^+(0,M, \Omega)$, the free boundary is denoted by 
$
\Gamma=\partial \Omega \cap B_r^+.
$ 
A blow-up limit of $\{u_j\} \subset P_1^+(0,M,\Omega)$ is a limit of the form $$\lim_{k \rightarrow \infty} \frac{u_{j_k}(s_kx)}{s_k^2},$$ where $\{j_k\}$ is a subsequence of $\{j\}$ and $s_k \rightarrow 0^+$.

\section{Non-transversal intersection and blow-ups}

In the following propositions, the class $P_1^+(0, M, \Omega)$ is in terms of a general $\Omega$ subject to the stated assumptions. 

\begin{prop} \cite[Proposition 3.6]{I} \label{th1}
Let $\{u_j\} \subset P_1^+(0,M, \Omega)$ and suppose $\{\nabla u_j \neq 0\} \cap \{x_n>0\} \subset \Omega$, $0 \in \overline{\{u_j \neq 0\}}$, and $\nabla u_j(0)=0$. Then one of the following is true:\\
(i) all blow-up limits of $\{u_j\}$ at the origin are of the form $u_0(x)=b x_n^2$ for some $b >0$;\\ 
(ii) there exists a blow-up limit of $\{u_j\}$ of the form $ax_1x_n+bx_n^2$ for $a \neq 0$, $b \in \mathbb{R}$.
\end{prop}

\begin{prop} \label{ke}
Suppose $\{u_j\} \subset P_1^+(0,M, \Omega)$. If $\{\nabla u_j \neq 0\} \cap \{x_2>0\} \subset \Omega$, $0 \in \overline{\{u_j \neq 0\}}$ and $\nabla u_j(0)=0$, then one of the following is true:\\
(i) all blow-up limits of $\{u_j\}$ at the origin are of the form $u_0(x)=bx_2^2$ for $b>0$;\\
(ii) there exists $\{u_{k_j}\} \subset \{u_j\}$, $j_1 \in \mathbb{N}$, and $s_j>0$ such that for all $j \ge j_1$, $$u_{k_j} \in C^{2,\alpha}(B_{s_j}^+).$$ 
\end{prop}

\begin{proof}
Either all blow-up limits are of the form $u_0(x)=bx_2^2$ or there exists a subsequence, say $$\tilde u_j(x)=\frac{u_{k_j}(r_jx)}{r_j^2},$$ producing a limit of the form $u_0(x)=ax_1x_2+bx_2^2$ for $a>0$ (up to a rotation). 
Pick $R \ge 1$ and note that since $\tilde u_j \rightarrow u_0$ in $C_{loc}^{1,\alpha}$, there exists $j_R \in \mathbb{N}$ such that for $j \ge j_R$, $|\nabla \tilde u_j|>\tilde c$ in 
$$E=\big(\{-R<x_1<-\frac{R}{2}\}\cup \{\frac{R}{2}<x_1<R\}\big)\cap B_{2R}^+,$$ and therefore, $\tilde u_j \in C^{2,\alpha}(E)$. In particular, up to a subsequence, 
$$\frac{|\partial_{x_1} \tilde u_j-ax_2|}{x_2}=\omega_j \rightarrow 0$$ so that
$$
\partial_{x_1} \tilde u_j \ge \frac{a}{2}x_2
$$ 
in $E$ for $j \ge \tilde j_R \in \mathbb{N}$. Now pick $\eta \in (0,R)$ and select $j_0'$ so that if $j \ge j_0'$, 
$$\partial_{x_1} \tilde u_j(x) \ge \frac{a \eta}{2}$$ 
for $x \in \{x_n \ge \eta \}$. Fix $s <\min\{\eta, \frac{R}{2}\}$ and suppose that for some $j \ge \max\{j_0', \tilde j_R\}$, 
$$\text{Int}\{ \nabla \tilde u_j =0 \} \cap B_s^+\neq \emptyset.$$ 
Let $S=\{0 < x_2 < \eta, -R< x_1 < R\}$ so that $v:=\partial_{x_1} \tilde u_j\ge 0$ on $\partial S$. By differentiating the equation, $Lv=0$ on $S \cap \Omega_j$, where $L$ is a linear second order uniformly elliptic operator. Since $v$ vanishes on $\partial \Omega_j$, it follows that $v>0$ in $S \cap \Omega_j$. Next note that for a disk $B \subset B_s^+$ such that $\nabla \tilde u_j =0$ in $B$, $u_j=m$ for some $m \in \mathbb{R}$ and there is a strip $\tilde S$ generated by translating the disk in the $x_1$-direction. Select another disk 
$$\tilde B \subset \tilde S \cap \{-R < x_1 < -\frac{R}{2}\}$$ 
and let $E_t=\tilde B+te_1$ for $t \in \mathbb{R}$. Since $v>0$ in $S \cap \Omega_j$ and $v=0$ in $\Omega_j^c$, it follows that $u < m$ on $\tilde B$. Denote by $t^*>0$ the first value for which $\partial E_t$ intersects $\{\tilde u_j=m\}$ and let $y \in \partial E_{t^*} \cap \{\tilde u_j=m\}$. Note that $F(D^2 \tilde u_j) \ge 0$ and $w=\tilde u_j-m<0$ in $E_{t^*}$ with $w(y)=0$, therefore by Hopf's principle (see e.g. \cite{MR2994551}), $\partial_n \tilde u_j(y) >0$. Therefore, there is $\mu>0$ such that $B_\mu(y) \subset \Omega_j$. In particular, $v>0$ on $B_\mu(y)$ and since $v = 0$ on $\Omega_j^c$, there is $p>0$ such that $\tilde u_j(y+e_1p)>m$ for $y+e_1p \in \{\tilde u_j = m\}$, a contradiction. Therefore, $\text{Int}\{ \nabla \tilde u_j =0 \} \cap B_s^+= \emptyset$ and non-degeneracy implies the claim.   
\end{proof}    

\begin{proof}[proof of Theorem \ref{tt}]
If not, then there exists $\epsilon>0$ such that for all $k \in \mathbb{N}$ there exists $u_k \in P_1^+(0,M, \Omega)$ with 
\begin{equation} \label{cont}
\Gamma(u_k) \cap B_{1/k}^+ \cap \mathcal{C}_\epsilon \neq \emptyset, 
\end{equation}
where $0 \in \overline{\Gamma(u_k)}.$ If all blow-ups of $\{u_k\}$ are half-space solutions, let $x_k \in \Gamma(u_k) \cap B_{1/k}^+ \cap \mathcal{C}_\epsilon$ and set $y_k=\frac{x_k}{r_k}$ with $r_k=|x_k|$. Consider $\tilde u_k(x)=\frac{u_k(r_kx)}{r_k^2}$ so that $y_k \in \Gamma(\tilde u_k)$, $\tilde u_k \rightarrow bx_2^2$, $y_{k} \rightarrow y \in \partial B_1 \cap C_\epsilon$ (up to a subsequence), and $y \in \Gamma(u_0)$, a contradiction. Therefore, Proposition \ref{ke} implies the existence of a subsequence $\{u_{k_j}\}$ of $\{u_k\}$ such that for all $j \ge j_1$, $u_{k_j} \in C^{2,\alpha}(B_{s_{j}}^+)$, where $j_1 \in \mathbb{N}$. Since $0 \in \overline{\Gamma(u_{k_j})},$ there exists $$x_j \in \Gamma(u_{k_j}) \cap B_{\frac{s_{j}}{2}}^+$$ which contradicts the continuity of $F$.   
\end{proof}

\begin{proof}[proof of Theorem \ref{cul}]
By Proposition \ref{ke}, either $u_0(x)=bx_2^2$ or $D^2u(0)$ exists and the rescaling of $u$ is given by $$u_j(x)=\frac{u(r_j x)}{r_j^2}=\langle x, D^2u(0)x\rangle + o(1).$$ Since $u_0(x_1,0)=0$, it follows that $u_0$ has the claimed form (up to a rotation). 
\end{proof}

\noindent \text{\bf Acknowledgments} I learned the moving ball technique from John Andersson and thank him for interesting discussions.

\bibliographystyle{alpha}

\bibliography{ngonref2}

\signei

\end{document}